\documentclass[leqno, 10pt, a4paper]{amsart}
\usepackage{amsmath,amssymb,latexsym}
\usepackage{color}

 \newtheorem{definition}{Definition}[section]
 \newtheorem{theorem}[definition]{Theorem}
 \newtheorem{lemma}[definition]{Lemma}
 \newtheorem{proposition}[definition]{Proposition}
 \newtheorem{corollary}[definition]{Corollary}

 \newtheorem*{theorem*}{Theorem}
\newtheorem*{proposition*}{Proposition}
\newtheorem*{lemma*}{Lemma}

 \theoremstyle{remark}

  \newtheorem*{acknowledgements}{Acknowledgements}

\newcommand{\op}[1]{\operatorname{#1}}

\newcommand{\tr}{\ensuremath{\op{tr}}}
\newcommand{\Tr}{\ensuremath{\op{Tr}}}
\newcommand{\Tra}{\ensuremath{\op{Trace}}}

\newcommand{\TR}{\ensuremath{\op{TR}}}

\newcommand{\Res}{\ensuremath{\op{Res}}}
\newcommand{\res}{\ensuremath{\op{res}}}

\newcommand{\Sp}{\op{Sp}}

\newcommand{\End}{\ensuremath{\op{End}}}

\newcommand{\C}{\ensuremath{\mathbb{C}}}
\newcommand{\N}{\ensuremath{\mathbb{N}}}
\newcommand{\R}{\ensuremath{\mathbb{R}}}
\newcommand{\Rn}{\ensuremath{\mathbb{R}^{n}}}
\newcommand{\URn}{\ensuremath{U\times \mathbb{R}^{n}}}

\newcommand{\URno}{\ensuremath{\mathbb{R}^{n}\setminus 0}}
\newcommand{\Z}{\ensuremath{\mathbb{Z}}}
\newcommand{\CZ}{\ensuremath{\mathbb{C}\setminus\mathbb{Z}}}

\newcommand{\cL}{\ensuremath{\mathcal{L}}}

\newcommand{\cW}{\ensuremath{\mathcal{W}}}

\newcommand{\psido}{$\Psi$DO}
\newcommand{\psidos}{$\Psi$DOs}

\newcommand{\pdoi}{\ensuremath{\Psi^{\text{int}}}}
\newcommand{\pdoz}{\ensuremath{\Psi^{\Z}}}
\newcommand{\pdocz}{\ensuremath{\Psi^{\CZ}}}

\newcommand{\ord}{{\op{ord}}}

\newcommand{\up}{\uparrow}
\newcommand{\down}{\downarrow}
\newcommand{\updown}{\uparrow \downarrow}

\newcommand{\zetaup}{\zeta_{\scriptscriptstyle{\uparrow}}}
\newcommand{\zetadown}{\zeta_{\scriptscriptstyle{\downarrow}}}
\newcommand{\zetaupdown}{\zeta_{\scriptscriptstyle{\updown}}}

\begin{document}

    \title{ON THE SINGULARITIES OF THE ZETA AND ETA FUNCTIONS OF AN ELLIPTIC OPERATOR}
 \author{Paul Loya}
 \address{Department of Mathematics, SUNY Binghamton University, Binghamton, NY, USA}
 \email{paul@math.binghamton.edu}
 \author{Sergiu Moroianu}
 \address{Institutul de Matematic\u{a} al Academiei Rom\^{a}ne, Bucharest, Romania}
 \email{moroianu@alum.mit.edu}
 \author{Rapha\"el Ponge}
\address{IPMU \& Graduate School of Mathematical Sciences, University of Tokyo, Tokyo, Japan}
\email{ponge@ms.u-tokyo.ac.jp}
\thanks{S.M.~was partially supported by grant PN-II-ID-PCE 1188 265/2009 (Romania).  R.P.~was partially
supported by JSPS Grant-in-Aid 30549291 (Japan).}

 \numberwithin{equation}{section}

 \begin{abstract}
Let $P$ be a selfadjoint elliptic operator of order $m>0$ acting on the sections of a Hermitian vector bundle over a
compact Riemannian manifold of dimension $n$. General arguments show that its zeta and eta functions may have poles only at
points of the form $s=\frac{k}{m}$, where $k$ ranges over all non-zero integers~$\leq n$. In this paper,
we construct elementary and explicit examples of perturbations of $P$ which make the zeta and eta functions become singular at
\emph{all} points at which they are allowed to have singularities. We proceed within three classes of operators: Dirac-type
operators, selfadjoint first-order differential operators, and selfadjoint elliptic pseudodifferential operators.
As consequences, we obtain genericity
results for the singularities of the zeta and eta functions in those settings. In particular, in the setting of Dirac-type
operators we obtain a purely analytical proof of a well known result of Branson-Gilkey~\cite{BG:REFODT}, which was 
obtained by invoking Riemannian invariant theory.
 As it turns out, the results of this paper contradict Theorem~6.3 of~\cite{Po:SAZFNCR}.  Corrections to that statement 
 are given in Appendix~\ref{sec:eta-odd-class}.
\end{abstract}

\maketitle

\section{Introduction}

Let $P:C^{\infty}(M,E)\rightarrow C^{\infty}(M,E)$ be a
selfadjoint elliptic \psido\ of order $m$, $m \in \N$, acting on the sections of a Hermitian vector bundle $E$ over a compact
Riemannian manifold $M^{n}$. The zeta and eta functions of $P$ are two of the most important spectral functions that
can be associated to $P$. Not only do they allow us to study the spectral properties of $P$ (including its spectral asymmetry), but when $P$ is a geometric operator 
(e.g., a Laplace-type or Dirac-type operator) they may carry a lot of a geometric information on the manifold $M$.

In particular, the residues of the zeta and eta functions are of special interest. More precisely, on the one hand, the residues at
integer points of the
zeta functions of the square of the Dirac operators are important in the context of noncommutative geometry (see, e.g.,
\cite{Ka:DOG}, \cite{Po:NCGLDVRG}). On the other hand, the residues of the eta function at integer points naturally come into in the index formula for Dirac
operators on manifolds with singularities (see~\cite{BS:ITFORSO}).

Recall that the zeta and eta functions of $P$ are obtained
as meromorphic continuations from the half-plane $\Re s>\frac{n}{m}$ to the whole complex plane of the functions,
\begin{equation*}
    \zeta(P;s) :=\sum_{\lambda \in \op{Sp} P\setminus \{0\}} \lambda^{-s} \qquad \text{and} \qquad
    \eta(P;s):=\sum_{\lambda \in \op{Sp} P\setminus \{0\}} \op{sign}(\lambda) |\lambda|^{-s},
\end{equation*}where each eigenvalue is repeated according to multiplicity and $\lambda^{-s}$ is defined in a
suitable way (there are actually several possible zeta functions; see~Section~\ref{sec.EFNCR}).

 General considerations show that the possible singularities of the meromorphic functions $\zeta(P;s)$ and $\eta(P;s)$ may only be simple
 poles contained in the admissible set,
 \begin{equation}
     \Sigma:= \biggl\{\frac{k}{m}; \ k \in \Z, \ k\leq n\biggr\}\setminus \left\{0\right\}.
     \label{eq:Intro.admissible-set}
\end{equation}
It then is natural to raise the following questions:
\begin{enumerate}
    \item[-] Given any point of the admissible set $\Sigma$, is there always an operator whose zeta or eta
    function is singular at that point?
    \item[-] Is there an operator whose zeta and eta functions are singular at  \emph{all} the points of the admissible set $\Sigma$?
    \item[-] If these phenomena do occur, how ``generic'' are they?
\end{enumerate}

In fact, if we endow the space $\Psi^{m}(M,E)$ with its standard Fr\'echet space topology, then we may expect that,
\emph{generically}, the zeta and eta function are singular at \emph{all} the points of the admissible set $\Sigma$.
To see this,  let us denote by $\Psi^{m}(M,E)^{\text{sa}}$ the real vector space of selfadjoint elements
of $\Psi^{m}(M,E)$ equipped with the induced topology, and consider a small open neighborhood $\cW$ of the origin in
$\Psi^{m}(M,E)^{\text{sa}}$ such that $P+R$ is elliptic for all $R\in \cW$. For $\sigma \in \Sigma$
define
\begin{equation*}
    \cW_{\sigma}:=\left\{R\in \cW; \ \Res_{s=\sigma}\zeta(P;s)\neq 0 \ \text{and} \ \Res_{s=\sigma}\eta(P+R;s)\neq 0\right\}.
\end{equation*}

Because the conditions $\Res_{s=\sigma}\zeta(P;s)\neq 0$ and
$\Res_{s=\sigma}\eta(P+R;s)\neq 0$ can be shown to be open conditions and reduce to simple conditions on the homogeneous
 components of the symbol of $R$ (if, say, $R$ is supported in a local chart), we may certainly expect
$\cW_{\sigma}$ to be a dense open subset of $\cW$. The Baire category theorem would then ensure
 us that $ \bigcap_{\sigma \in \Sigma}\cW_{\sigma}$ is dense, so that we could find an arbitrary small
 selfadjoint perturbation of $P$ that makes the zeta and eta functions become singular at all points of the
 admissible set $\Sigma$.

 In the light of this, one would like to have \emph{explicit} constructions of such perturbations, not only in the category of pseudodifferential
 operators, but also by restricting ourselves to the smaller categories of differential operators and Dirac-type operators.
 The aim of this note is to exhibit those types of perturbations within the following classes of operators:
 \begin{enumerate}
     \item[(i)] First order selfadjoint elliptic differential operators (see~Proposition~\ref{prop:1stDO.perturbation}).

     \item[(ii)] Dirac-type operators  (see~Corollary~\ref{cor:Dirac.generic-sing}).

     \item[(iii)] Selfadjoint elliptic pseudo-differential operators 
     (see Proposition~\ref{prop:PsiDOs1-sing-eta} in the first-order case and
     Proposition~\ref{prop:higherPsiDOs-sing-eta1} in the general case).
 \end{enumerate}

The proofs of those results are based on straightforward considerations on the noncommutative residue
trace of Guillemin~\cite{Gu:NPWF} and Wodzicki~\cite{Wo:NCRF}. This trace provides us with a neat and unified treatment of the
residues of the zeta and eta functions.

The result for the category (ii) of Dirac operators is an immediate by-product of the results for the category (i).
In the odd-dimensional case this provides us with an alternative proof of a
well-known result of Branson-Gilkey~\cite[Theorem 4.3]{BG:REFODT} about the genericity of the singularities of the eta function of a
Dirac-type operator at positive odd integers~$\leq n$ in even dimension and at all non-zero even integers~$\leq n$ in odd 
dimension (those are the only points at which the eta function may have singularities; see below).  

While Branson-Gilkey's result was obtained by making use of Riemannian invariant theory, and 
hence was specific to Dirac-type operators, our 
argument is purely analytic and applies to general first order selfadjoint elliptic differential operators. It would be interesting 
to obtain the genericity result of Branson-Gilkey for  general first order selfadjoint 
elliptic differential operators in odd dimension by purely analytical means. 

As it turns out,  when $P$ is a first order selfadjoint elliptic differential operator, the admissible sets at which the
zeta and eta functions is much smaller than $\Sigma$ (see Section~\ref{sec:1stDOs}). In particular, they are subsets of the set of positive
integers~$\leq n$.  As a result, for such an operator, it is enough to use the
simplest possible perturbations, namely, perturbations by real constants,
\begin{equation*}
   P+a, \qquad a \in\R.
\end{equation*}

In order to deal with first order \psidos\ the perturbations by constants are too crude, so we need to refine them. We consider
perturbations of the form,
\begin{equation}
   P_{\epsilon}:=P+\epsilon |P|, \qquad P_{\epsilon,c}:=P_{\epsilon}+cF(P)|P_{\epsilon}|^{-n},
   \label{eq:Intro.Pec}
\end{equation}where $\epsilon\in (-1,1)$ and $c\geq 0$ are small enough, and we have denoted by $|P|$ the absolute value of $P$ and by
$F(P)=P|P|^{-1}$ its sign operator.

The results when $P$ is a \psido\ of order~$m>1$ follow from for the first order case by considering perturbations of the form,
\begin{equation*}
    P_{\epsilon,c,a}=F(Q_{\epsilon,c}+a)|Q_{\epsilon,c}+a|^{m},
\end{equation*}where $a$, $c$ and $\epsilon$ are as above and $Q_{\epsilon,c}$ is defined as in~(\ref{eq:Intro.Pec}) by
substituting the first order operator $Q:=F(P)|P|^{\frac{1}{m}}$ for $P$.

This note is organized as follows. In Section~\ref{sec.EFNCR}, we recall some background on  the noncommutative residue
 trace and  the zeta and eta functions of an elliptic operator.  In Section~\ref{sec:1stDOs}, we prove the results for first order differential
 operators and Dirac-type operators. In Section~\ref{sec:PDOs1}, we deal with first-order pseudo-differential operators. In
 Section~\ref{sec:higher-orderPsiDOs}, we derive the results for higher order pseudo-differential operators.

  In addition, two appendices are included. In Appendix~\ref{sec:Topology-PsiDOs}, for the reader's convenience, we recall
  the construction of the standard Fr\'echet space topologies on the spaces of \psidos\ of given orders.
 In Appendix~\ref{sec:eta-odd-class}, we include corrections to Theorem~6.3 of~\cite{Po:SAZFNCR},
 which misstates that if $P$ is odd-class and its order has opposite
parity to $\dim M$, then its eta function is regular at \emph{all} integer points. As we shall explain, this property does
hold true in odd dimension, but in even dimension we can only obtain regularity at \emph{even} integer points
(cf.~Theorem~\ref{thm:odd-class.new}).

\begin{acknowledgements}
  The authors wish to thank Peter Gilkey for fruitful discussions regarding the subject matter of this paper.
\end{acknowledgements}

\noindent \emph{Notation.} Throughout this paper we let $M$ be a compact Riemannian manifold of dimension $n$ and $E$ be a Hermitian vector bundle
over $M$ of rank $r$.
The Riemannian metric of $M$ and the Hermitian metric of $E$ then allow us to define an inner product on $L^{2}(M,E)$ with
respect to which we define adjoints.

In addition, for $q \in \C$, we denote by $\Psi^{q}(M,E)$ the space of (classical) \psidos\ of order $q$
acting on the sections of $E$. We shall endow $\Psi^{q}(M,E)$ with its standard Fr\'echet space topology, whose definition is recalled in
Appendix~\ref{sec:Topology-PsiDOs}.

\section{The noncommutative residue and the eta and zeta functions}\label{sec.EFNCR}
In this section, we recall the basic definitions and properties of the noncommutative residue trace
and the zeta and eta functions of an elliptic operator.

 Define
 \begin{gather*}
     \pdoi(M,E):=\bigcup_{\Re  q<-n}\Psi^{q}(M,E) , \\
  \pdoz(M,E):=\bigcup_{q\in \Z}\Psi^{q}(M,E), \qquad    \pdocz(M,E):=\bigcup_{q \in \CZ} \Psi^{q}(M,E).
 \end{gather*}Notice that, as $M$ is compact, the class $\pdoz(M,E)$ is an algebra.

 Any operator $P\in
 \pdoi(M,E)$ is trace-class. Moreover, its Schwartz kernel $k_{P}(x,y)$ is continuous and the restriction to the diagonal
 $k_{P}(x,x)$ gives rise to a smooth $\End E$-valued density, i.e., a smooth section of the bundle
 $|\Lambda|(M)\otimes \End E$, where we denote by $|\Lambda|(M)$ the line bundle of densities over $M$.  Therefore,
 \begin{equation*}
     \Tra( P) =\int_{M} \tr_{E_{x}}k_{P}(x,x).
 \end{equation*}

 The map $P\rightarrow k_{P}(x,x)$ is holomorphic from $\pdoi(M,E)$ to the Fr\'echet space
 $C^{\infty}(M,|\Lambda|(M)\otimes \End E)$, where holomorphy is meant with respect to the notion of
 holomorphic families of \psidos\ as defined in~\cite{Gu:RTCAFIO} and~\cite{Ok:CHTEORTF}. Namely, if $(P(z))_{z\in \C}$
 is any holomorphic family of \psidos\ with values in $\pdoi(M,E)$, then $z \rightarrow k_{P}(x,x)$ is a holomorphic map
 from $\C$ to $C^{\infty}(M,|\Lambda|(M)\otimes \End E)$.

 As observed by Guillemin~\cite{Gu:GLD} and Wodzicki~\cite{Wo:NCRF} (see also~\cite{KV:GDEO}),
 the map  $P\rightarrow
 k_{P}(x,x)$ has a unique holomorphic extension as a
  map $P\rightarrow t_{P}(x)$ from $\pdocz(M,E)$ to
 $C^{\infty}(M,|\Lambda|(M)\otimes \End E)$.  Therefore,
 the operator trace can be uniquely extended into the holomorphic functional on $\pdocz(M,E)$
defined by
 \begin{equation*}
     \TR P :=\int_{M}\tr_{E_{x}}t_{P}(x) \qquad \forall P\in \pdocz(M,E).
 \end{equation*}

 In addition, the functional $\TR$ induces on the algebra $\pdoz(M,E)$ the noncommutative residue trace of
Guillemin~\cite{Gu:NPWF} and Wodzicki~\cite{Wo:NCRF} as follows.

Let $P\in \pdoz(M,E)$ and, given some local coordinates and trivialization of $E$, let $p(x,\xi)\sim \sum_{j\geq
0}p_{m-j}(x,\xi)$ be the symbol of $P$ with respect to these local coordinates and trivialization. Define
\begin{equation}
    c_{P}(x)=(2\pi)^{-n}\int_{S^{n-1}}p_{-n}(x,\xi)d^{n-1}\xi,
    \label{eq:Background.residue-density}
\end{equation}where $p_{-n}(x,\xi)$ is the symbol of degree $-n$ of $P$ and $d^{n-1}\xi$ denotes the surface measure of
the sphere $S^{n-1}$.

At first glance, the function $c_{P}(x)$ depends on the choice of the local coordinates and
trivialization. However, it can be shown that if we interpret it as a density, then it makes sense intrinsically on $M$
as a smooth $\End E$-valued density, i.e., as an element of $C^{\infty}(M,|\Lambda|(M)\otimes \End E)$.

Furthermore, if $P\in \pdoz(M,E)$ and $(P(z))_{z\in \C}$ is a holomorphic family of \psidos\ such that $P(0)=P$ and $\ord P(z)=z+\ord P$,
then, near $z=0$, the map $z \rightarrow t_{P(z)}(x)$ has at worst a simple pole singularity such that
\begin{equation*}
    \Res_{z=0}t_{P(z)}(x)=-c_{P}(x).
\end{equation*}

The \emph{noncommutative residue} is the linear functional  on $\pdoz(M,E)$ defined by
\begin{equation}
    \Res P:=\int_{M}\tr_{E_{x}}c_{P}(x) \qquad \forall P\in \pdoz(M,E).
     \label{eq:Background.NCR}
\end{equation}

\begin{proposition}[\cite{Gu:RTCAFIO}, \cite{Wo:NCRF}] \label{prop:Background.NCR}
    The noncommutative residue has the following properties:
    \begin{enumerate}
        \item  It vanishes on $\Psi^{-(n+1)}(M,E)$ and on all differential operators.
        In particular, it vanishes on all smoothing
        operators.\smallskip

        \item  It is a trace on the algebra $\pdoz(M,E)$, i.e.,
        \begin{equation*}
            \Res \left[P_{1}P_{2}\right]=\Res\left[ P_{2}P_{1}\right] \qquad \forall P_{j}\in \pdoz(M,E).
        \end{equation*}

        \item  Let $P\in \pdoz(M,E)$ and let $(P(z))_{z\in \C}$ be a holomorphic family of \psidos\ such that $P(0)=P$
        and $\ord P(z)=z+\ord P$. Then, near $z=0$, the map $z \rightarrow \TR P(z)$ has at worst a simple pole
        singularity such that
        \begin{equation*}
                     \Res_{z=0} \TR P(z)=-\Res P.
\end{equation*}
    \end{enumerate}
\end{proposition}

In addition, if $M$ is connected, then a result of Wodzicki~\cite{Wo:PhD} (see also~\cite{LP:UMDEPDO}, \cite{Po:TPDOSC})
states that the noncommutative residue is the unique trace up to constant multiples on the algebra
$\pdoz(M,E)$.

Next, let $P:C^{\infty}(M,E)\rightarrow C^{\infty}(M,E)$ be a selfadjoint elliptic \psido\ of order $m$, $m \in \N$.
Considered as an
unbounded operator on $L^{2}(M,E)$ with domain the Sobolev space $L^{2}_{m}(M,E)$, the operator $P$ is closed,
selfadjoint and has a compact resolvent. Therefore, the spectral decomposition of $P$ allows us to define a (Borel) functional
calculus for $P$. In particular, for any $s \in \C$, this enables us to define the complex powers $P_{\up}^{s}$ and
$P_{\down}^{s}$ by functional calculus associated to the functions,
\begin{gather*}
    \lambda_{\up}^{s}:=e^{is\arg_{\up}s}, \qquad  \lambda \in \C\setminus i[0,\infty),   \\
     \lambda_{\down}^{s}:=e^{is\arg_{\down}s}, \qquad  \lambda \in \C\setminus -i[0,\infty),
\end{gather*}
where $\arg_{\up}$ (resp., $\arg_{\down}$) is the continuous determination of the argument with values in
$(-\frac{3\pi}{2},\frac{\pi}{2})$ (resp., $(-\frac{\pi}{2},\frac{3\pi}{2})$). In particular,
\begin{equation}
    P^{0}_{\updown}=1-\Pi_{0}(P), \qquad P_{\updown}^{k}=P^{k} \quad \forall k \in \Z\setminus 0,
    \label{eq:NCR.Pudpdown-integers}
\end{equation}where $\Pi_{0}(P)$ is the orthogonal projection onto $\ker P$.
It follows from the results of Seeley~\cite{Se:CPEO} that $P_{\up}^{s}$ and $P_{\down}^{s}$ are \psidos\ of order
$ms$ and the families $(P_{\up}^{s})_{s \in \C}$ and $(P_{\down}^{s})_{s \in \C}$ are holomorphic families of 
\psidos~(see also~\cite{Ok:CHTEORTF} concerning the holomorphy of these families).  

The zeta functions of $P$ are the meromorphic functions on $\C$
defined by
\begin{equation*}
    \zetaup(P;s):=\TR P_{\updown}^{-s} \qquad \text{and} \qquad  \zetadown(P;s):=\TR P_{\updown}^{-s}.
\end{equation*}
In particular, for $\Re s>\frac{n}{m}$,
\begin{equation*}
    \zetaupdown(P;s)=\Tr  P_{\updown}^{-s} =\sum_{\lambda \in \op{Sp}(P)\setminus 0} \lambda_{\updown}^{-s},
\end{equation*}where each non-zero eigenvalue of $P$ is repeated according to multiplicity.

It follows from Proposition~\ref{prop:Background.NCR} that the functions $\zetaupdown(P;s)$ are analytic outside
\begin{equation}
    \Sigma:=\biggl\{\frac{k}{m}; \ k\in \Z, \ k\leq n \biggr\}\setminus \{0\}.
    \label{eq:Non-generic.singular-set}
\end{equation}
Moreover, near any $\sigma\in \Sigma$,  there is at worst a simple pole singularity such that
        \begin{equation}
            m .\Res_{s=\sigma}\zeta_{\updown}(P;s)= \Res P_{\updown}^{-\sigma}.
             \label{eq:Background.Residues-zeta}
         \end{equation}
Notice that at $s=0$ we have
\begin{equation*}
    m .\Res_{s=0}\zeta_{\updown}(P;s)= \Res P_{\updown}^{0}=\Res \left[ 1-\Pi_{0}(P)\right]=0,
\end{equation*}since the noncommutative residue vanishes on differential and smoothing operators.

When $P$ is positive, i.e., its spectrum is contained in $[0,\infty)$, the zeta functions $\zetaup(P;s)$ and
$\zetadown(P;s)$ agree, so in that case we shall drop the subscripts $\updown$ and use the single notation
$\zeta(P;s)$.  For instance, the absolute value $|P|:=\sqrt{P^{2}}$ is a positive selfadjoint elliptic \psido\ of order $m$. Thus, for this operator, we have a single
zeta function,
\begin{equation*}
    \zeta(|P|;s):=\TR |P|^{-s}.
\end{equation*}Notice that some authors refer to $\zeta(|P|;s)$ as the zeta function of $P$ (see, e.g.,~\cite{Gi:ITHEASIT}).

Like the other zeta functions $\zetaupdown(P;s)$, the function $\zeta(|P|;s)$ is analytic on $\C\setminus \Sigma$ and it has
at worst simple pole singularities on $\Sigma$ such that
        \begin{equation}
            m .\Res_{s=\sigma}\zeta(|P|;s)= \Res |P|^{-\sigma} \qquad \forall \sigma \in \Sigma.
             \label{eq:Background.Residues-zeta|P|}
         \end{equation}

The eta function of a selfadjoint elliptic \psido\ was introduced by Atiyah-Patodi-Singer~\cite{APS:SARG1} as a tool to
study the spectral asymmetry of such an operator. It can be defined as follows.

Let $F:=P|P|^{-1}$ be the sign operator of $P$. This is a selfadjoint elliptic \psido\ of order zero such that
\begin{equation}
    F^{2}=1-\Pi_{0}(P).
    \label{eq:NCR.squareF}
\end{equation}
The eta function of $P$ is the meromorphic function on $\C$ defined by
\begin{equation*}
    \eta(P;s):=\TR \left[F|P|^{-s}\right].
\end{equation*}In particular, for $\Re s>\frac{n}{m}$,
 \begin{equation*}
     \eta(P;s)=\Tr \left[F|P|^{-s}\right]= \sum_{\lambda \in \op{Sp}(P)\setminus 0} \op{sign}(\lambda)|\lambda|^{-s},
 \end{equation*}where each non-zero eigenvalue of $P$ is repeated according to multiplicity.

Like the zeta functions, the eta function is analytic on $\C\setminus \Sigma$ and  has
at worst simple pole singularities on $\Sigma$ such that
        \begin{equation}
             m.\Res_{s=\sigma}\eta(P;s)=\Res \left[F|P|^{-\sigma}\right] \qquad \forall \sigma \in \Sigma.
             \label{eq:Background.Residues-eta}
         \end{equation}
Notice that at $s=0$ we  have
\begin{equation*}
   m . \Res_{s=\sigma}\eta(P;s)= \Res \left[F|P|^{0}\right] =\Res F.
\end{equation*}However, important results of Atiyah-Patodi-Singer~\cite{APS:SARG3} and Gilkey~\cite{Gi:RGEFO} ensure us
that $\Res F$ is zero, so that $\eta(P;s)$ is always regular at the origin.

\section{First-order differential operators}\label{sec:1stDOs}
In this section, we shall look at the singularities of the eta and zeta functions of  first-order elliptic differential operators and
Dirac-type operators. The main results of this section will be consequences of
elementary lemmas, which we shall state for pseudodifferential operators since we will need to use of them in the
next sections.

Let $P:C^{\infty}(M,E)\rightarrow C^{\infty}(M,E)$ be a selfadjoint elliptic first-order \psido\ and let us look at the
effects of the simplest possible perturbations, namely, perturbations by constants,
  \begin{equation*}
    P+a, \qquad a \in \R.
 \end{equation*}The operator $P+a$ is a selfadjoint elliptic first-order \psido\ with the same principal symbol
 as $P$.

In the sequel, given $Q \in \Psi^{q}(M,E)$, $q \in \C$, and $Q_{j}\in \Psi^{q-j}(M,E)$, $j=0,1,\cdots$, we
shall write
\begin{equation*}
    Q\simeq \sum_{j\geq 0} Q_{j}
\end{equation*}to mean that
\begin{equation*}
    Q-\sum_{j<N}Q_{j}\in \Psi^{q-N}(M,E) \qquad \forall N\in \N.
\end{equation*}

\begin{lemma}\label{lem:Non-generic.Binomial}
   Let $a \in \R$.  Then, for all $k \in \N$,
    \begin{equation}
        (P+a)^{-k}\simeq \sum_{j\geq 0} \binom{-k}{j} a^{j}P^{-(k+j)}.
        \label{eq:Non-generic.Binomial}
    \end{equation}
\end{lemma}
\begin{proof} Let $\Pi_{0}(P)$ be the orthogonal projection onto $\ker P$. As $\Pi_{0}(P)$ is a smoothing operator, we
    have
   \begin{equation*}
       P+a=P(1+aP^{-1})+a\Pi_{0}(P)=P(1+aP^{-1}) \quad \bmod \Psi^{-\infty}(M,E).
   \end{equation*}
   Let $N\in \N$ and set $u=aP^{-1}$. Then $u$ is contained in $\Psi^{-1}(M,E)$ and
   \begin{equation*}
       (1+u).\sum_{0\leq j<N}(-u)^{j}=1-(-u)^{N}=1 \quad \bmod \Psi^{-N}(M,E).
   \end{equation*}Thus,
   \begin{equation*}
       (P+a)\sum_{0\leq j<N}(-u)^{j}=P(1+u).\sum_{0\leq j<N}(-u)^{j}=P \quad \bmod \Psi^{1-N}(M,E).
   \end{equation*}Multiplying these equalities by $P^{-1}(P+a)^{-1}$ then gives
   \begin{equation}
       (P+a)^{-1}= P^{-1}.\sum_{0\leq j<N}(-u)^{j}\quad \bmod \Psi^{-1-N}(M,E).
        \label{eq:Non-generic.Pa-1.P-1}
   \end{equation}

   Let $k \in \N$. From~(\ref{eq:Non-generic.Pa-1.P-1}) we get
   \begin{equation}
       (P+a)^{-k}=P^{-k}. \biggl( \sum_{0\leq j<N}(-u)^{j}\biggr)^{k}\quad \bmod \Psi^{-k-N}(M,E).
       \label{eq:Non-generic.Pa-k.P-k}
   \end{equation}
   Observe that, near $x = 0$,
   \begin{equation*}
       \biggl( \sum_{0\leq j<N}(-x)^{j}\biggr)^{k}=\left( 1-(-x)^{N}\right)^{k}(1+x)^{-k} = \sum_{0\leq
       j<N}\binom{-k}{j}x^{j} +\op{O}(x^{N}).
   \end{equation*}Since $ \biggl( \sum_{0\leq j<N}(-x)^{j}\biggr)^{k}$ is a polynomial, we deduce that
   \begin{equation*}
        \biggl( \sum_{0\leq j<N}(-x)^{j}\biggr)^{k}=\sum_{0\leq
       j<N}\binom{-k}{j}x^{j} \quad \bmod  \op{Span}\left\{ x^{N},\cdots x^{k(N-1)}\right\}.
   \end{equation*}It then follows that
   \begin{align*}
       \biggl( \sum_{0\leq j<N}(-u)^{j}\biggr)^{k} &= \sum_{0\leq
       j<N}\binom{-k}{j}u^{j} \quad \bmod  \op{Span}\left\{ u^{N},\cdots u^{k(N-1)}\right\}\\
       &= \sum_{0\leq
       j<N}\binom{-k}{j}a^{j}P^{-j} \quad \bmod \Psi^{-N}(M,E).
   \end{align*}
   Combining this with~(\ref{eq:Non-generic.Pa-k.P-k}) we get
   \begin{equation*}
       (P+a)^{-k}= \sum_{0\leq
       j<N}\binom{-k}{j}a^{j}P^{-(k+j)}\quad  \bmod \Psi^{-k-N}(M,E),
   \end{equation*}which proves~(\ref{eq:Non-generic.Binomial}). The proof is
   complete.
\end{proof}

\begin{lemma}\label{lem:DiffO.Res|P|n}
    The noncommutative residue $\Res |P|^{-n}$ is always~$>0$.
\end{lemma}
\begin{proof}
Let $p_{1}(x,\xi)$ be the principal symbol of $P$. Then $|P|^{-n}$ is a \psido\ of order $-n$ with principal symbol
 $|p_{1}(x,\xi)|^{-n}$. This symbol is positive-definite everywhere, so the density $c_{|P|^{-n}}(x)$ defined
 by~(\ref{eq:Background.residue-density}) takes values in positive-definite sections of $\End E$, and hence the scalar density
 $\tr_{E}c_{|P|^{-n}}(x)$ is~$>0$. Combining this with~(\ref{eq:Background.NCR}) shows that $\Res |P|^{-n}$ is~$>0$,
 proving the lemma.
\end{proof}

In the sequel, we let $F:=P|P|^{-1}$ be the sign of $P$. We observe that~(\ref{eq:NCR.squareF}) implies that, for all $k \in \Z$,
\begin{equation}
    F^{k}= \left\{
    \begin{array}{ll}
        F & \text{if $k$ is odd},  \\
        1-\Pi_{0}(P)=1 \bmod \Psi^{-\infty}(M,E) &  \text{if $k$ is even.}
    \end{array}\right.
 \label{eq:DiffO-signPk}
\end{equation}

\begin{lemma}\label{lem:DiffO.signP+a}
Let  $a \in \R$ and denote by $F_{a}$ the sign of $P+a$. Then
    \begin{equation}
        F_{a}=F \quad \bmod \Psi^{-\infty}(M,E).
        \label{eq:DiffO.sign-P+a}
    \end{equation}
\end{lemma}
\begin{proof}
For $\lambda \in \C$, we shall denote by $\Pi_{\lambda}(P)$ the orthogonal projection onto $\ker (P-\lambda)$. If $\lambda
\not \in \Sp P$, then $\Pi_{\lambda}(P)=0$, but if $\lambda \in \Sp P$, then $\Pi_{\lambda}(P)$ is a finite-rank
projection and a smoothing operator.
Moreover, with respect to the strong topology of $\cL(L^{2}(M,E))$, we have
\begin{equation*}
    F=\sum_{\lambda>0}\Pi_{\lambda}(P) - \sum_{\lambda<0}\Pi_{\lambda}(P) \qquad \text{and} \qquad
    F_{a}=\sum_{\lambda>-a}\Pi_{\lambda}(P) - \sum_{\lambda<-a}\Pi_{\lambda}(P) .
\end{equation*}

If $a>0$, then
\begin{multline}
    F_{a}=\sum_{\lambda>0}\Pi_{\lambda}(P) + \sum_{0\geq \lambda >-a}\Pi_{\lambda}(P) - \biggl( \sum_{\lambda<0}\Pi_{\lambda}(P) -
    \sum_{-a\leq \lambda <0} \Pi_{\lambda}(P)\biggr)\\
    = F+2 \sum_{-a <\lambda <0}\Pi_{\lambda}(P) + \Pi_{0}(P)+\Pi_{-a}(P).
    \label{eq:Non-generic.Fa}
\end{multline}
Similarly, if $a<0$, then
\begin{equation*}
    F_{a}= F-2 \sum_{0 <\lambda <-a}\Pi_{\lambda}(P) - \Pi_{0}(P)-\Pi_{-a}(P).
\end{equation*}
In any case $F_{a}$ and $F$ agree up to a smoothing operator. The lemma is proved.
\end{proof}

In the remainder of the section we assume that $P$ is a first order \emph{differential operator}. As explained in the
previous section, the zeta and eta functions of $P$ may have singularities only at the points of the admissible set,
\begin{equation*}
    \Sigma:=\biggl\{k\in \Z; \ k\leq n\biggr\}\setminus \left\{0\right\}.
\end{equation*}

As we shall now see, the fact that $P$ is a differential operator allows us to shrink the admissible sets at which
the zeta and eta functions may have singularities.

\begin{lemma} If $M$ has even dimension $n$, then
\begin{enumerate}
    \item[(i)] $\zeta_{\up}(P;s)$ and $\zeta_{\down}(P;s)$ may have singularities only at positive integers~$\leq n$.\smallskip

    \item[(ii)] $\zeta(|P|;s)$ may have singularities only at even positive integers~$\leq n$.\smallskip

    \item[(iii)]  $\eta(P;s)$ may have singularities only at odd positive integers~$\leq n$.
\end{enumerate}
\end{lemma}
\begin{proof}
Let $k$ be a non-zero integer~$\leq n$. In view of~(\ref{eq:NCR.Pudpdown-integers}) and (\ref{eq:Background.Residues-zeta}) we have
\begin{equation}
    \Res_{s=k}\zetaupdown(P;s)= \Res P^{-k}_{\updown}=\Res P^{-k}.
    \label{eq:DO.residues-zeta}
\end{equation}
Moreover, using~(\ref{eq:Background.Residues-zeta|P|}) and~(\ref{eq:DiffO-signPk}) we get
\begin{equation}
    \Res_{s=k}\zeta(|P|;s)= \Res |P|^{-k}=\Res\left[ F^{k}P^{-k}\right]= \left\{
    \begin{array}{ll}
       \Res P^{-k}   & \text{if $k$ is even},  \\
        \Res\left[ FP^{-k}\right] &  \text{if $k$ is odd.} \\
    \end{array}\right.
     \label{eq:DO.Residues-zeta|P|}
\end{equation}
Likewise, combining~(\ref{eq:Background.Residues-eta}) and~(\ref{eq:DiffO-signPk}) gives
\begin{equation}
    \Res_{s=k}\eta(P;s)= \Res\left[F |P|^{-k}\right]=\Res\left[ F^{k+1}P^{-k}\right]= \left\{
    \begin{array}{ll}
       \Res\left[ FP^{-k} \right]  & \text{if $k$ is even},  \\
        \Res P^{-k} &  \text{if $k$ is odd.} \\
    \end{array}\right.
     \label{eq:DO.Residues-eta}
\end{equation}

As the noncommutative residue of a differential operator is always zero, we see that $\Res P^{-k}=0$ whenever $k$ is
negative. Combining this with~(\ref{eq:DO.residues-zeta}) shows that $\zetaupdown(P;s)$ cannot have poles at negative integers, and hence
may have poles only at positive integers~$\leq n$. In addition, by using~(\ref{eq:DO.Residues-zeta|P|})--(\ref{eq:DO.Residues-eta})
we also see that $\zeta(|P|;s)$ (resp., $\eta(P;s)$)
cannot have poles at even (resp., odd) negative integers.

Let $\Pi_{-}(P)$ be the orthogonal projection onto the negative eigenspace of $P$. As $P$ is a differential operator of odd order
and $M$ has even dimension, it is proved on~\cite[page~1081]{Po:SAZFNCR} that
\begin{equation*}
    \Res \left[\Pi_{-}(P)P^{-k}\right]=\frac{1}{2}\Res P^{-k}.
\end{equation*}
As $FP^{-k}=(1-2\Pi_{-}(P))P^{-k}$ we see that
\begin{equation*}
    \Res \left[FP^{-k}\right]=\Res \left[ (1-2\Pi_{-}(P))P^{-k}\right]=\Res P^{-k} -2\Res \left[
    \Pi_{-}(P)P^{-k}\right]=0.
\end{equation*}
Combining this with~(\ref{eq:DO.Residues-zeta|P|})--(\ref{eq:DO.Residues-eta}) shows that the function $\zeta(|P|;s)$ (resp., $\eta(P;s)$)
cannot have poles at odd (resp., even) integers.

It follows from all this that the function $\zeta(|P|;s)$ (resp., $\eta(P;s)$) may have poles only at even (resp., odd) positive
integers~$\leq n$. This completes the proof.
\end{proof}

\begin{lemma}\label{lem:DO.singularities-odd-dimension}
    If $M$ has odd dimension $n$, then
\begin{enumerate}
    \item[(i)] $\zeta_{\up}(P;s)$ and $\zeta_{\down}(P;s)$ are entire functions.\smallskip

    \item[(ii)] $\zeta(|P|;s)$ may have singularities only at odd integers~$\leq n$.\smallskip

    \item[(iii)]  $\eta(P;s)$ may have singularities only at non-zero even integers~$\leq n$.
\end{enumerate}
\end{lemma}
\begin{proof}
   In the terminology of~\cite{KV:GDEO}, an operator $Q\in \Psi^{m}(M,E)$, $m\in \Z$, is said to be odd-class if, in any
   local coordinates and local trivialization of $E$, the symbol $q(x,\xi)\sim\sum_{j\geq 0}q_{m-j}(x,\xi)$ of $Q$
   satisfies
   \begin{equation}
       q_{m-j}(x,-\xi)=(-1)^{m-j}q_{m-j}(x,\xi) \qquad \forall j \in \N_{0}.
       \label{eq:DO.odd-class-condition}
   \end{equation}Using~(\ref{eq:Background.residue-density}) it is not difficult to check that the above condition for $m-j=-n$ implies that the
   density $c_{Q}(x)=0$ if $n$ is odd (cf.~\cite{KV:GDEO}). Thus the noncommutative residue of an odd-class \psido\ is always zero in odd
   dimension.

   Any differential operator is odd-class and it not difficult to check that the parametrix of any elliptic odd-class
   \psido\ is again odd-class. Thus, for all $k\in \Z$, the operator $P^{-k}$ is odd-class. As $n$ is odd, we then
   deduce that
   \begin{equation*}
       \Res P^{-k}=0 \qquad \forall k\in \Z.
   \end{equation*}
   Combining this with~(\ref{eq:DO.residues-zeta})--(\ref{eq:DO.Residues-eta}) (and the regularity of $\eta(P;s)$ at 
   $s=0$) yields the lemma.
\end{proof}

Recaling that $M$ has dimension $n$, we are now in a position to prove the main result of this section. 

\begin{proposition}\label{prop:1stDO.perturbation}   Suppose that $P$ is a selfadjoint elliptic first-order differential operator. Then, for all but  finitely many values of $a$,
   \begin{enumerate}
       \item $\zeta_{\up}(P+a;s)$ and $\zeta_{\down}(P+a;s)$ are singular at all positive integers~$\leq n$ if $n$ is even.

       \item $\zeta(|P+a|;s)$  is singular at all positive even (resp., odd) integers~$\leq n$ if $n$ is even
       (resp., odd).

       \item $\eta(P+a;s)$  is singular at all positive odd (resp., even) integers~$\leq n$ if $n$ is even
       (resp., odd).
   \end{enumerate}
\end{proposition}
\begin{proof}
Let $a \in \R$ and let $k$ be a positive integer~$\leq n$. In view of~(\ref{eq:DO.residues-zeta}) we have
\begin{equation}
    \Res_{s=k}\zeta_{\updown}(P+a;s)=\Res (P+a)^{-k}.
    \label{eq:DO.residues-zeta-P+a}
\end{equation}

Let $F_{a}:=(P+a)|P+a|^{-1}$ be the sign operator of $P+a$. In addition, set $l=0$ if $n$ is even and $l=1$ if $n$ is
odd. Then~(\ref{eq:DO.Residues-zeta|P|}) shows that, if $k$ and $n$ have the same parity, then
\begin{equation}
    \Res_{s=k}\zeta(|P+a|;s)= \Res \left[ F_{a}^{l}(P+a)^{-k}\right].
\end{equation}
Likewise, using~(\ref{eq:DO.Residues-eta}) we see that, if $k$ and $n$ have opposite parities, then
\begin{equation}
    \Res_{s=k}\eta(P+a;s)= \Res \left[ F_{a}^{l}(P+a)^{-k}\right].
\label{eq:DO.Residues-eta-P+a}
\end{equation}

It follows from~(\ref{eq:DO.residues-zeta-P+a})--(\ref{eq:DO.Residues-eta-P+a}) that in order to prove the proposition it is enough to show, that for any positive
integer $k\leq n$ and for all but finitely values of $a$,
\begin{equation}
    \Res \left[ F_{a}^{l}(P+a)^{-k}\right]\neq 0.
\end{equation}

Thanks to Lemma~\ref{lem:Non-generic.Binomial} we know that
\begin{equation*}
    (P+a)^{-k}= \sum_{0\leq j \leq n-k}\binom{-k}{j}a^{j}P^{-(k+j)} \qquad \bmod \Psi^{-(n+1)}(M,E),
\end{equation*}Combining this with Lemma~\ref{lem:DiffO.signP+a} we get
\begin{align*}
   F_{a}^{l}(P+a)^{-k}  & = F^{l}(P+a)^{-k} \quad \bmod  \Psi^{-\infty}(M,E),\\
     & = \sum_{0\leq j \leq n-k}\binom{-k}{j}a^{j} F^{l}P^{-(k+j)} \quad \bmod \Psi^{-(n+1)}(M,E).
\end{align*}
As the noncommutative residue vanishes on
$\Psi^{-(n+1)}(M,E)$, we deduce that
\begin{equation}
     \Res\left[ F_{a}^{l} (P+a)^{-k}\right]= \sum_{0\leq  j \leq n-k}\binom{-k}{j}a^{j}\Res\left[F^{l} P^{-(k+j)}\right].
     \label{eq:DiffO.Resk-etaP+a}
\end{equation}
Thus $\Res\left[ F_{a}^{l} (P+a)^{-k}\right]$ is polynomial in $a$ of degree $n-k$ whose leading coefficient is a non-zero multiple of $\Res
\left[ F^{l}P^{-n}\right]$.

If $n$ is even, then $F^{l}P^{-n}=P^{-n}=|P|^{-n}$. If $n$ is odd, then~(\ref{eq:DiffO-signPk}) implies that
$F^{l}P^{-n}=F^{n}P^{-n}=|P|^{-n}$. Combining this with Lemma~\ref{lem:DiffO.Res|P|n} shows that, in both cases, \[\Res \left[F^{l}P^{-n}\right]=\Res |P|^{-n}>0.\]
It then follows that $ \Res\left[ F_{a}^{l} (P+a)^{-k}\right]$ is a \emph{non-zero} polynomial, and so
it may vanish for at most finitely many values of $a$. This completes the proof.
\end{proof}

\begin{corollary}\label{cor:1stDO. generic-sing}
    For a generic selfadjoint elliptic first-order differential  operator $P:C^{\infty}(M,E)\rightarrow
    C^{\infty}(M,E)$ the following hold:
    \begin{enumerate}
       \item[(i)] $\zeta_{\up}(P;s)$ and $\zeta_{\down}(P;s)$ are singular at all positive integers~$\leq n$ if $n$ is even.

       \item[(ii)] $\zeta(|P|;s)$  is singular at all positive even (resp., odd) integers~$\leq n$ if $n$ is even
       (resp., odd).\smallskip

       \item[(iii)] $\eta(P;s)$  is singular at all positive odd (resp., even) integers~$\leq n$ if $n$ is even
       (resp., odd).
    \end{enumerate}
\end{corollary}

Finally, recall that a Dirac-type operator is a selfadjoint first order differential operator $D:C^{\infty}(M,E)\rightarrow
C^{\infty}(M,E)$ such that the principal symbol of $D^{2}$ satisfies
\begin{equation*}
    \sigma_{2}(D^{2})(x,\xi)=|\xi|^{2}.\op{id}_{E_{x}} \qquad \forall (x,\xi)\in T^{*}M,
\end{equation*}where $|\xi|^{2}:=g^{ij}(x)\xi_{i}\xi_{j}$ is the Riemannian metric of $T^{*}M$.

By a result of Branson-Gilkey~\cite[Theorem 4.3]{BG:REFODT}, if we restrict ourselves to the class of Dirac-type operators
then, \emph{generically}, the eta function has singularities at all positive odd integers~$k<n$ if
$n$ is even and at all non-zero even integers~$<n$ if $n$ if odd. 

The proof of Branson-Gilkey's result relied on the Riemannian 
invariant theory of~\cite{ABP:OHEIT} and~\cite{Gi:ITHEASIT}. While the use of the Riemannian invariant theory is an extremely 
powerfool tool to get precised information on the coefficients of the heat kernel asymptotics, it is rather specific to 
Laplace-type and Dirac-type operators. As a result, the arguments of~\cite{BG:REFODT} do not extend to general 
differential operators. Therefore, it would be desirable to have a purely analytical proof of Branson-Gilkey's result. 

Observe that the class of Dirac-type operator is invariant under perturbations by constants. Therefore, specializing Proposition~\ref{prop:1stDO.perturbation}
to Dirac-type operators we immediately get

\begin{corollary}\label{cor:Dirac.generic-sing}
    Let $D:C^{\infty}(M,E)\rightarrow
C^{\infty}(M,E)$ be a Dirac-type operator. Then we always can perturbate $D$ by an arbitrary
small real constant $a$ so that
    \begin{enumerate}
       \item[(i)] $\zeta_{\up}(D;s)$ and $\zeta_{\down}(D;s)$ are singular at all positive integers~$\leq n$ if $n$ is even.\smallskip

       \item[(ii)] $\zeta(|D|;s)$  is singular at all positive even (resp., odd) integers~$\leq n$ if $n$ is even
       (resp., odd).\smallskip

       \item[(iii)] $\eta(D;s)$  is singular at all positive odd (resp., even) integers~$\leq n$ if $n$ is even
       (resp., odd).
    \end{enumerate}
\end{corollary}

In particular, in the even-dimensional case this provides us with a purely analytical proof of the aforementioned genericity result of Branson and 
Gilkey. 

\section{First-order pseudodifferential operators}\label{sec:PDOs1}
In this section, we let $P:C^{\infty}(M,E)\rightarrow C^{\infty}(M,E)$ be a selfadjoint elliptic first-order  \psido.
Then the set of admissible points at which the eta function $\eta(P;s)$ and the zeta functions
$\zeta_{\updown}(P;s)$ and $\zeta(|P|;s)$ are allowed to have singularities is
\begin{equation}
     \Sigma= \left\{ k\in \Z; \ k\leq n\right\}\setminus \left\{0\right\}.
     \label{eq:1stPsiDOs.admissible-set}
\end{equation}
That is, $\Sigma$ consists of all non-zero integers~$\leq n$.

In the sequel, we let $F:=P|P|^{-1}$ the sign operator of $P$. 

\begin{lemma}\label{lem:Non-generic.residues-eta-Pa2}
 If the following three conditions simultaneously hold
 \begin{equation}
    \Res \left[ F|P|^{-n}\right] \neq 0, \qquad \Res P\neq 0, \qquad \Res |P|\neq 0,
     \label{eq:Non-generic.conditions-sing-Pa}
\end{equation}then, for all but countably many values of the real number $a$, the function $\eta(P+a;s)$ and the zeta functions
$\zeta_{\updown}(P+a;s)$ and $\zeta(|P+a|;s)$ are singular at {all} non-zero integers~$\leq n$.
 \end{lemma}
\begin{proof}
Let $a \in \R$ and denote by $F_{a}=(P+a)|P+a|^{-1}$ the sign operator of $P+a$. In addition, let
$k$ be a non-zero integer~$\leq n$. The formulas~(\ref{eq:DO.residues-zeta})--(\ref{eq:DO.Residues-eta}) remain valid for $P+a$. Therefore, each of the residues at
$s=k$ of the zeta functions $\zeta_{\updown}(P+a;s)$ and $\zeta(|P+a|;s)$ and the eta function $\eta(P;s)$ is equal to
one of the following
\begin{equation*}
    \Res (P+a)^{-k} \qquad \text{or} \qquad \Res \left[F_{a}(P+a)^{-k}\right].
\end{equation*}
Furthermore, as by Lemma~\ref{lem:DiffO.signP+a} the operators $F$ and $F_{a}$ agree up to a smoothing operator and the noncommutative
residue vanishes on smoothing operators, we get
\begin{equation*}
   \Res \left[F_{a}(P+a)^{-k}\right]= \Res \left[F(P+a)^{-k}\right].
\end{equation*}
Thus, in order to prove the lemma it is enough to show that if the conditions~(\ref{eq:Non-generic.conditions-sing-Pa}) hold, then,
for each non-zero integer~$k\leq n$, all but
finitely values of $a$ satisfy
\begin{equation}
    \Res (P+a)^{-k}\neq 0 \qquad \text{and} \qquad \Res \left[ F(P+a)^{-k}\right]\neq 0.
    \label{eq:1stPsiDOs.conditions-singularities-eta-zetaP+a}
\end{equation}

Let $l\in \{0,1\}$. If $k$ is a positive integer~$\leq n$, then, as in~(\ref{eq:DiffO.Resk-etaP+a}), we have
\begin{equation*}
     \Res\left[F^{l} (P+a)^{-k}\right]= \sum_{0\leq j \leq n-k}\binom{-k}{j}a^{j}\Res\left[ F^{l} P^{-(k+j)} \right].
\end{equation*}Combining this with~(\ref{eq:DiffO-signPk}) we see that $ \Res\left[F^{l} (P+a)^{-k}\right]$ is a polynomial in $a$ whose leading coefficient is a non-zero
multiple of
\begin{equation*}
    \Res \left[F^{l}P^{-n}\right] = \Res \left [F^{l+n}|P|^{-n}\right] = \left\{
    \begin{array}
    {ll}

        \Res |P|^{-n} & \text{if $n=l \bmod 2$},  \\

        \Res\left[F|P|^{-n}\right] &  \text{if $n=l+1 \bmod 2$}. \\

    \end{array}\right.
\end{equation*}

If $k$ is a negative integer, then the binomial formula gives
\begin{equation*}
   F^{l} (P+a)^{-k}= F^{l}(P+a)^{|k|}=  \sum_{0\leq j \leq |k|}\binom{|k|}{j}a^{|k|-j}F^{l}P^{j},
\end{equation*}Combining this with the fact that $\Res [1]=\Res F=0$ (cf.~the end of Section~\ref{sec.EFNCR}), we obtain
\begin{equation*}
      \Res\left[F^{l} (P+a)^{-k}\right]= \sum_{0\leq j \leq |k|}\binom{|k|}{j}a^{|k|-j}\Res \left[F^{l} P^{j} \right]=
      \sum_{1\leq j \leq |k|}\binom{|k|}{j}a^{|k|-j}\Res\left[F P^{j}\right].
\end{equation*}This shows that $ \Res\left[F^{l} (P+a)^{-k}\right]$ is a polynomial in $a$  of degree $|k|-1$ whose leading
coefficient is equal to
\begin{equation*}
    \Res \left[F^{l}P\right]=\left\{
    \begin{array}
    {ll}
         \Res P & \text{if $l=0$},  \\
        \Res \left[FP\right]=\Res|P| & \text{if $l=1$}.
    \end{array}\right.
\end{equation*}

It follows from all this that, for any non-zero integer $k\leq n$, both $\Res (P+a)^{-k}$ and $\Res\left[F (P+a)^{-k}\right]$ are polynomials in $a$
whose leading coefficients are non-zero multiples of one of  the following noncommutative residues:
\begin{equation*}
  \Res |P|^{-n}, \qquad  \Res \left[ F|P|^{-n}\right], \qquad \Res P, \qquad \Res |P|.
\end{equation*}
Thanks to Lemma~\ref{lem:DiffO.Res|P|n} we know that the first of these noncommutative residues is always non-zero. Therefore, if the other three are
 non-zero, then $\Res (P+a)^{-k}$ and $
\Res\left[F (P+a)^{-k}\right]$ both are \emph{non-zero} polynomials in $a$, and hence vanish for at most finitely many values
of $a$. As mentioned above, this proves the lemma.
\end{proof}

We shall now construct perturbations of $P$ that ensure us that the three conditions in~(\ref{eq:Non-generic.conditions-sing-Pa}) hold. To this end, for $-1<\epsilon<1$ and
$c\geq 0$, we define
\begin{equation}
    P_{\epsilon}:= P+\epsilon |P|, \qquad P_{\epsilon,c}:=P_{\epsilon}+cF|P_{\epsilon}|^{-n}.
    \label{eq:Non-generic.PePec}
\end{equation}
The operators $P_{\epsilon}$ and $P_{\epsilon,c}$ are selfadjoint elliptic first-order \psidos\ with the same principal
symbols. 

\begin{lemma}\label{lem:Non-generic.conditions-etaPa-Pec}
    If $\epsilon$ is small enough, then, for all but finitely many values of $c$, the operator $P_{\epsilon,c}$ satisfy
    the three conditions~(\ref{eq:Non-generic.conditions-sing-Pa}).
\end{lemma}
\begin{proof}
    In the sequel, we denote by $\Pi_{0}(P)$ the orthogonal projection onto $\ker P$ and by $\Pi_{+}(P)$ (resp.,
    $\Pi_{-}(P)$) the orthogonal projection onto the positive (resp., negative) eigenspaces of $P$. Thus,
    \begin{gather}
        \Pi_{+}(P)+\Pi_{-}(P)=1-\Pi_{0}(P), \qquad F=\Pi_{+}(P)-\Pi_{-}(P),
         \label{eq:Generic.Pi0pm}\\
        \Pi_{\pm}(P)=\frac{1}{2}(1\pm F)-\frac{1}{2}\Pi_{0}(P).
       \label{eq:Generic.Pi0pmb}
    \end{gather}
    In particular, Eq.~(\ref{eq:Generic.Pi0pmb}) shows that $\Pi_{+}(P)$ and $\Pi_{-}(P)$ are zeroth order \psidos.

    We shall use similar notations for the corresponding projections associated to $P_{\epsilon}$ and $P_{\epsilon,c}$. In
    addition, we let $F_{\epsilon}:=P_{\epsilon}|P_{\epsilon}|^{-1}$ and
    $F_{\epsilon,c}:=P_{\epsilon,c}|P_{\epsilon,c}|^{-1}$ be the respective sign operators of $P_{\epsilon}$ and
    $P_{\epsilon,c}$.

   By definition,
    \begin{equation*}
        P_{\epsilon}=P+\epsilon|P|=(1+\epsilon F)P.
    \end{equation*}
    Thus $P$ and $P_{\epsilon}$ have same null space, and if $\lambda \in \op{Sp}P\setminus \{0\}$,
    then $(1+\epsilon \op{sign}(\lambda))\lambda$ is an eigenvalue of $P_{\epsilon}$ with same
    eigenspace and sign as $\lambda$. Therefore,
    \begin{equation}
      \Pi_{0}(P_{\epsilon})=\Pi_{0}(P), \qquad  \Pi_{\pm}(P_{\epsilon})=\Pi_{\pm}(P), \qquad F_{\epsilon}=F.
      \label{Non-generic.spectral-projections-Pe}
    \end{equation}It then follows that
    \begin{gather}
        |P_{\epsilon}|=F_{\epsilon}P_{\epsilon}=F(1+\epsilon F)P=(1+\epsilon F)|P|,  \\
        F|P_{\epsilon}|^{-n}=F(1+\epsilon F)^{-n}|P|^{-n}.\label{eq:Non-generic.FPe|Pe|-n}
    \end{gather}

   Let $\lambda \in \Sp P$. If $\lambda>0$ (resp., $\lambda <0$), then $(1+ \epsilon )^{-n}$ (resp., $(1- \epsilon
   )^{-n})$ is an eigenvalue of $(1+\epsilon F)^{-n}$ with eigenspace $\ker (P-\lambda)$. If $\lambda=0$, then $1$ is
   an eigenvalue of $(1+ \epsilon F)^{-n}$ with eigenspace $\ker P$. It follows that
   \begin{equation*}
      (1+ \epsilon F)^{-n} = (1+\epsilon)^{-n}\Pi_{+}(P) + (1-\epsilon)^{-n}\Pi_{-}(P)+\Pi_{0}(P).
  \end{equation*}

  Let $u(\epsilon)$ and $v(\epsilon)$ be the functions on $(-1,1)$ defined by
  \begin{equation*}
      u(\epsilon)=\frac{1}{2}\left( (1+\epsilon)^{-n}+(1-\epsilon)^{-n}\right) \quad  \text{and}\quad v(\epsilon)=\frac{1}{2}\left(
      (1+\epsilon)^{-n}-(1-\epsilon)^{-n}\right),
  \end{equation*}so that $ (1\pm \epsilon)^{-n}=u(\epsilon)\pm v(\epsilon)$.
  Then $(1+\epsilon)^{-n}\Pi_{+}(P) + (1-\epsilon)^{-n}\Pi_{-}(P)$ is equal to
  \begin{multline}
        \left( u(\epsilon)+v(\epsilon)\right) \Pi_{+}(P)+
      \left( u(\epsilon)-v(\epsilon)\right) \Pi_{-}(P) \\ =
      u(\epsilon) \left( \Pi_{+}(P)+\Pi_{-}(P)\right) + v(\epsilon)\left( \Pi_{+}(P)-\Pi_{-}(P)\right) \\
      =u(\epsilon)\left( 1-\Pi_{0}(P)\right) +v(\epsilon)F.
      \label{eq:Non-generic.ueve}
  \end{multline}

  Combining~(\ref{eq:Non-generic.FPe|Pe|-n}) and~(\ref{eq:Non-generic.ueve}) we get
  \begin{equation}
      F|P_{\epsilon}|^{-n}=F\left( u(\epsilon)\left( 1-\Pi_{0}(P)\right)
      +v(\epsilon)F\right)|P|^{-n}=u(\epsilon)F|P|^{-n}+v(\epsilon)|P|^{-n}.
      \label{eq:Non-generic.F|Pe|-n}
  \end{equation}Therefore, as $\epsilon \rightarrow 0$, we have
  \begin{align*}
      \Res \left[ F|P_{\epsilon}|^{-n}\right] &= u(\epsilon)\Res\left[  F|P|^{-n}\right]  +v(\epsilon)\Res |P|^{-n}\\ &= \left(1
      +\op{O}(\epsilon^{2})\right) \Res\left[  F|P|^{-n}\right]  + \left( -n\epsilon +\op{O}(\epsilon^{3})\right) \Res 
      |P|^{-n} \\ &=
      \Res\left[  F|P|^{-n}\right] -n\epsilon \Res |P|^{-n} +\op{O}(\epsilon^{2}).
  \end{align*}By Lemma~\ref{lem:DiffO.Res|P|n} the noncommutative residue $\Res |P|^{-n}$ is always~$>0$, so
  we see that if $\epsilon$ is small enough, then $\Res F|P_{\epsilon}|^{-n}$ is non-zero.

  From now on we choose $\epsilon$ so that $\Res  F|P_{\epsilon}|^{-n}\neq 0$. In view of~(\ref{eq:Non-generic.PePec}) we have
  \begin{equation*}
      \Res P_{\epsilon,c}=\Res P_{\epsilon}+c\Res F|P_{\epsilon}|^{-n},
  \end{equation*}so we see that for all values of $c$, except maybe one,  the noncommutative residue $ \Res
  P_{\epsilon,c}$ is non-zero.

  The operators $P_{\epsilon}$ and $P_{\epsilon,c}=P_{\epsilon}+ cF|P_{\epsilon}|^{-n}$ have the same null
  space. If $\lambda$ is a non-zero eigenvalue of $P_{\epsilon}$, then
  $\lambda+c \op{sign}(\lambda)|\lambda|^{-n}$ is an eigenvalue of
  $P_{\epsilon,c}$ with eigenspace $\ker (P_{\epsilon}-\lambda)$ and the same sign
  as $\lambda$. Combining this with~(\ref{Non-generic.spectral-projections-Pe}) shows that
  \begin{equation}
   \Pi_{0}(P_{\epsilon,c})=\Pi_{0}(P_{\epsilon})=\Pi_{0}(P) \qquad \text{and} \qquad   F_{\epsilon,c}=F_{\epsilon}=F.
   \label{Non-generic.spectral-projections-Pec}
  \end{equation}Therefore,  
  \[
  |P_{\epsilon,c}|=F_{\epsilon,c}P_{\epsilon,c}=F_{\epsilon}\left(P_{\epsilon}+cF_{\epsilon}|P_{\epsilon}|^{-n}\right)=
     |P_{\epsilon}|+c|P_{\epsilon}|^{-n},
 \] where we used \eqref{eq:NCR.squareF}, and hence
  \begin{equation*}
      \Res |P_{\epsilon,c}|= \Res  |P_{\epsilon}| + c \Res |P_{\epsilon}|^{-n}.
  \end{equation*}Since $\Res |P_{\epsilon}|^{-n}\neq 0$, we see that for all, except maybe one value of $c$, the
  noncommutative residue $ \Res |P_{\epsilon,c}|$ is non-zero.

  Since $P_{\epsilon}$ and $P_{\epsilon,c}$ have the same principal symbols, the principal symbols of $|P_{\epsilon}|^{-n}$ and
  $|P_{\epsilon,c}|^{-n}$ agree, i.e., those
 operators differ by a \psido\ of order~$\leq -(n+1)$. As $F_{\epsilon,c}=F$, we see that
 $F_{\epsilon,c}|P_{\epsilon,c}|^{-n}$ and $F|P_{\epsilon}|^{-n}$ agree modulo an element in
 $\Psi^{-(n+1)}(M,E)$, and hence $\Res F_{\epsilon,c}|P_{\epsilon,c}|^{-n} =\Res
 F|P_{\epsilon}|^{-n}\neq 0$.

It follows from all this that if  $\epsilon$ is small enough, then, for all but finitely many values of $c$,
 all three conditions in~(\ref{eq:Non-generic.conditions-sing-Pa}) are satisfied by $P_{\epsilon,c}$. The lemma is proved.
\end{proof}

Combining Lemma~\ref{lem:Non-generic.residues-eta-Pa2} and Lemma~\ref{lem:Non-generic.conditions-etaPa-Pec} we immediately get

\begin{proposition}\label{prop:PsiDOs1-sing-eta}
    If $\epsilon$ is small enough, then, for all but countably many values of $a$ and all but finitely many values of
    $c$, the eta function $\eta(P_{\epsilon,c}+a;s)$ and the zeta functions
$\zeta_{\updown}(P_{\epsilon,c}+a;s)$ and $\zeta(|P_{\epsilon,c}+a|;s)$ are singular at {all} the non-zero integers~$\leq
n$, that is, at all points of the admissible set~(\ref{eq:1stPsiDOs.admissible-set}).
\end{proposition}

Observe that from~(\ref{eq:Non-generic.PePec}) and~(\ref{eq:Non-generic.F|Pe|-n}) we get
\begin{gather*}
    P_{\epsilon,c} + a=P+T_{\epsilon,c,a},\\
    T_{\epsilon,c,a}:=\epsilon |P|+cF|P_{\epsilon}|^{-n} +a= \epsilon |P|+ cu(\epsilon)F|P|^{-n}+cv(\epsilon)|P|^{-n}+a.
\end{gather*}Clearly, as $(\epsilon,c,a)\rightarrow (0,0,0)$ the operator $T_{\epsilon,c,a}$ converges to $0$ in
$\Psi^{1}(M,E)$. Therefore, we obtain

\begin{corollary}\label{cor:PsiDOs1-sing-eta}
    We always can perturbate $P$ by an arbitrary small selfadjoint element of $\Psi^{1}(M,E)$ in such way that the eta
    function $\eta(P;s)$ and the zeta functions
$\zeta_{\updown}(P;s)$ and $\zeta(|P|;s)$ become singular at all points of the admissible set~(\ref{eq:1stPsiDOs.admissible-set}).
\end{corollary}

\section{General Case}\label{sec:higher-orderPsiDOs}
In this section, we extend the results of the previous section to operators of higher order. In the sequel, if
$P:C^{\infty}(M,E)\rightarrow C^{\infty}(M,E)$ is any given selfadjoint elliptic \psido, then we shall denote
by $F(P)$ its sign.

Throughout this section we let $P:C^{\infty}(M,E)\rightarrow C^{\infty}(M,E)$ be a selfadjoint order elliptic \psido\
of order $m$, $m \in \N$. Then the set of admissible points at which the eta and zeta functions of $P$ are allowed to have
singularities is
\begin{equation}
    \Sigma:=\biggl\{\frac{k}{m}; \ k \in \Z, \ k\leq n \biggr\}\setminus \{0\}.
    \label{eq:general.admissible-set}
\end{equation}

Let $Q:C^{\infty}(M,E)\rightarrow C^{\infty}(M,E)$ be the operator defined by
\begin{equation*}
    Q:=F(P)|P|^{\frac{1}{m}}.
\end{equation*}Then $Q$ is a selfadjoint elliptic first-order \psido\  and we have
\begin{equation}
    F(Q)=F(P) , \qquad |P|=|Q|^{m}, \qquad P=F(Q)|Q|^{m}.
    \label{eq:Non-generic.Q-P}
\end{equation}

\begin{lemma}\label{lem:General.residues-eta-zeta-PQ}
    Let $k$ be a non-zero integer~$\leq n$. Then the following are equivalent:
    \begin{enumerate}
        \item[(i)] The eta function $\eta(P;s)$ and the zeta functions $\zeta_{\updown}(P;s)$ and $\zeta(|P|;s)$ are
        singular at $s=\frac{k}{m}$.

        \item[(ii)]   The eta function $\eta(Q;s)$ and the zeta functions $\zeta_{\updown}(Q;s)$ and $\zeta(|Q|;s)$ are
        singular at $s=k$.

        \item[(iii)] Both $ \Res |P|^{-\frac{k}{m}}$ and $ \Res\left[
        F(P)|P|^{-\frac{k}{m}}\right]$ are non-zero.
    \end{enumerate}
\end{lemma}
\begin{proof}
    Set $F=F(P)$. As $F=F(Q)$ and $|P|^{\frac{k}{m}}=|Q|^{k}=F(Q)^{k}Q^{k}$, by \eqref{eq:DiffO-signPk} we see that (iii) is equivalent to having
    \begin{equation*}
         \Res Q^{-k}\neq 0 \qquad \text{and} \qquad  \Res\left[ F(Q)Q^{-k}\right]\neq 0.
    \end{equation*}These conditions are exactly the conditions
    in~(\ref{eq:1stPsiDOs.conditions-singularities-eta-zetaP+a}) for $Q$ in the special case $a=0$.
    Thus, the equivalence
    between (ii) and (iii) follows from the beginning of the proof of
    Lemma~\ref{lem:Non-generic.residues-eta-Pa2}.

    It remains to prove that (i) and (iii) are equivalent.  Using~(\ref{eq:Background.Residues-zeta|P|})
    and~(\ref{eq:Background.Residues-eta}) we get
\begin{equation*}
    m\Res_{s=\frac{k}{m}}\zeta(|P|;s)= \Res |P|^{-\frac{k}{m}} ,\quad
    m\Res_{s=\frac{k}{m}}\eta(P;s)= \Res\left[ F|P|^{-\frac{k}{m}}\right].
\end{equation*}Therefore $\eta(P;s)$ and $\zeta(|P|;s)$ are
        singular at $s=\frac{k}{m}$ if and only if (iii) holds.

        To complete the proof it enough to show that if (iii) holds, then the zeta functions $\zeta_{\up}(P;s)$ and $\zeta_{\down}(P;s)$ too are
        singular at $s=\frac{k}{m}$.

        Set $\sigma =\frac{k}{m}$. By the very definition of the powers $P_{\updown}^{-\sigma}$
        (cf.~Section~\ref{sec.EFNCR}), we have
\begin{gather*}
  P_{\uparrow}^{-\sigma}=\Pi_{+}(P)|P|^{-\sigma}+e^{i\pi \sigma}\Pi_{-}(P)|P|^{-\sigma}, \\
      P_{\downarrow}^{-\sigma}=\Pi_{+}(P)|P|^{-\sigma}+e^{-i\pi \sigma}\Pi_{-}(P)|P|^{-\sigma},
\end{gather*}where $\Pi_{+}(P)$ (resp.,
    $\Pi_{-}(P)$) is the orthogonal projection onto the positive (resp., negative) eigenspaces of $P$.  Combining this
    with~(\ref{eq:Generic.Pi0pmb}) we get
    \begin{align*}
        P_{\up}^{-\sigma} & = \left\{\frac{1}{2}(1+ F)-\Pi_{0}(P)\right\}|P|^{-\sigma}+e^{i\pi \sigma}\left\{\frac{1}{2}(1- F)-\Pi_{0}(P)\right\}|P|^{-\sigma} \\
        & =\frac{1}{2}(1+e^{i\pi \sigma})|P|^{-\sigma} + \frac{1}{2}(1-e^{i\pi \sigma})F|P|^{-\sigma}.
    \end{align*}
    As  $\frac{1}{2}(1\pm e^{i\pi \sigma})= e^{i\pi \frac{\sigma}{2}} \frac{1}{2}(e^{-i\pi \frac{\sigma}{2}} \pm e^{i\pi
\frac{\sigma}{2}})$, we  obtain
\begin{equation*}
    P_{\up}^{-\sigma}= e^{i\pi \frac{\sigma}{2}} \left\{ \cos\left(\frac{\pi \sigma}{2} \right)|P|^{-\sigma} -i \sin
    \left(\frac{\pi \sigma}{2} \right)
        F|P|^{-\sigma}\right\}.
\end{equation*}
Combining this with~(\ref{eq:Background.Residues-zeta}) then gives
\begin{align}
    m.\Res_{s=\sigma} \zeta_{\up}(P;s) & = \Res P_{\up}^{-\sigma} \nonumber \\ &= e^{ i\pi \frac{\sigma}{2}} \left\{
    \cos\left(\frac{\pi \sigma}{2} \right)
    \Res \left[ |P|^{-\sigma}\right] - i \sin \left(\frac{\pi \sigma}{2} \right)
       \Res \left[  F|P|^{-\sigma}\right]\right\}.
       \label{eq:General.residues-zetaup}
\end{align}

We claim that $\Res |P|^{-\sigma}$ and $ \Res \left[  F|P|^{-\sigma}\right]$ are real numbers. Indeed,  as for
$\Re s>\frac{n}{m}$ we have
$\Tra \left[ |P|^{-s}\right] = \overline{ \Tra \left[ (|P|^{-s})^{*}\right] }= \overline{ \Tra \left[
    |P|^{-\overline{s}}\right] }$,  the meromorphic functions $\TR |P|^{-s} $ and $ \overline{ \TR
   |P|^{-\overline{s}} }$ agree. Combining this with~(\ref{eq:Background.Residues-zeta|P|}) then shows that
\begin{equation*}
 \Res  |P|^{-\sigma}=m.\Res_{s=\sigma}  \TR |P|^{-s} =  m.\Res_{s=\sigma} \overline{ \TR
   P|^{-\overline{s}} }= \overline{\Res  |P|^{-\sigma}},
\end{equation*}proving that $\Res |P|^{-\sigma}$ is a real number. A similar argument shows that $ \Res \left[
F|P|^{-\sigma}\right]$ too is a real number.

Since $\Res  |P|^{-\sigma}$ and $ \Res \left[  F|P|^{-\sigma}\right]$ are real numbers, Eq.~(\ref{eq:General.residues-zetaup})
shows that the real and imaginary parts of $e^{- i\pi \frac{\sigma}{2}}  \Res_{s=\sigma} \zeta_{\up}(P;s)$ are
given by
\begin{gather*}
    m .\Re \left(e^{-i\pi \frac{\sigma}{2}}  \Res_{s=\sigma} \zeta_{\up}(P;s)\right)=  \cos\left(\frac{\pi
    \sigma}{2} \right)
    \Res  |P|^{-\sigma},  \\
     m .\Im \left(e^{-i\pi \frac{\sigma}{2}}  \Res_{s=\sigma} \zeta_{\up}(P;s)\right)= - \sin(\frac{\pi \sigma}{2} )
    \Res \left[ F|P|^{-\sigma}\right] .
\end{gather*}Since $\cos(\frac{\pi \sigma}{2} ) $ and $\sin(\frac{\pi \sigma}{2} ) $ cannot be simultaneously zero, we
deduce that if both $\Res  |P|^{-\sigma}$ and $\Res \left[ F|P|^{-\sigma}\right] $ are  non-zero, then
$\Res_{s=\sigma} \zeta_{\up}(P;s)$ must be non-zero. This means that if (iii) holds, then
$\zeta_{\up}(P;s)$ is
        singular at $s=\frac{k}{m}$.

        A similar argument shows that if (iii) holds, then the zeta function $\zeta_{\down}(P;s)$ too is
        singular at $s=\frac{k}{m}$. This completes the proof.
\end{proof}

For $(\epsilon,c)\in (-1,1)\times [0,\infty)$ define $Q_{\epsilon}$ and $Q_{\epsilon,c}$ as
in~(\ref{eq:Non-generic.PePec}), i.e.,
\begin{equation*}
    Q_{\epsilon}:= Q+\epsilon |Q|, \qquad Q_{\epsilon,c}:=Q_{\epsilon}+cF(Q)|Q_{\epsilon}|^{-n}.
\end{equation*}
Let $a\in \R$. The operator $Q_{\epsilon,c}+a$ is a selfadjoint elliptic first-order \psido, so we define a selfadjoint elliptic \psido\ of order $m$ by letting
\begin{equation*}
    P_{\epsilon,c,a}:=F(Q_{\epsilon,c}+a)|Q_{\epsilon,c}+a|^{m}.
\end{equation*}

\begin{proposition}\label{prop:higherPsiDOs-sing-eta1}
 \begin{enumerate}
     \item  We can make $P_{\epsilon,c,a}$ be arbitrary close to $P$ by picking up any sufficiently small data
     $(\epsilon,c,a)$.

     \item  If $\epsilon$ is small enough, then, for all but countably many values of $a$ and all but finitely many values of
    $c$, the eta function $\eta(P_{\epsilon,c,a};s)$ and the zeta functions $\zeta_{\updown}(P_{\epsilon,c,a};s)$ and $\zeta(|P_{\epsilon,c,a}|;s)$
    are singular at all points of the admissible set~(\ref{eq:general.admissible-set}).  \end{enumerate}
\end{proposition}
\begin{proof}
  As $F(P_{\epsilon,c,a})=F(Q_{\epsilon,c}+a)$ and  $|P_{\epsilon,c,a}|=|Q_{\epsilon,c}+a|^{m}$, we have
  \begin{equation*}
      Q_{\epsilon,c}+a=F(Q_{\epsilon,c}+a)|Q_{\epsilon,c}+a|=F(P_{\epsilon,c,a})|P_{\epsilon,c,a}|^{\frac{1}{m}}.
  \end{equation*}Therefore, Lemma~\ref{lem:General.residues-eta-zeta-PQ} shows that, for any non-zero integer $k\leq n$, the following are equivalent:
  \begin{enumerate}
      \item[(i)]  The eta function $\eta(P_{\epsilon,c,a};s)$ and the zeta functions $\zeta_{\updown}(P_{\epsilon,c,a};s)$ and $\zeta(|P_{\epsilon,c,a}|;s)$
    are singular at $s=\frac{k}{m}$.

      \item[(ii)] The eta function $\eta(Q_{\epsilon,c}+a;s)$ and the zeta functions $\zeta_{\updown}(Q_{\epsilon,c}+a;s)$ and $\zeta(|Q_{\epsilon,c}+a|;s)$
    are singular at $s=k$.
  \end{enumerate}

  Moreover, Proposition~\ref{prop:PsiDOs1-sing-eta} asserts that, if $\epsilon$ is small enough, then, for all but countably many values of $a$
  and all but finitely many values of
    $c$, the eta function $\eta(Q_{\epsilon,c}+a;s)$ and the zeta functions $\zeta_{\updown}(Q_{\epsilon,c}+a;s)$ and $\zeta(|Q_{\epsilon,c}+a|;s)$
    are  singular at all non-zero integers~$\leq n$. Combining this with the equivalence of
    (i) and (ii) yields the 2nd part of the proposition.

In order to prove the 1st part of the proposition we have to show that $P_{\epsilon,c,a}$ converges to $P$ as $(\epsilon,c,a)$ goes to
$(0,0,0)$. The only (minor) difficulty comes
from the term $F(Q_{\epsilon,c} + a)$ which does not depend continuously on the data  $(\epsilon,c,a)$.

Let $f_{\epsilon,c}:\R \rightarrow \R$ be the function defined by
\begin{equation*}
    f_{\epsilon,c}(\lambda)=\lambda+\epsilon
|\lambda|+c\op{sign}(\lambda)\, | \lambda + \epsilon |\lambda|  |^{-n} \quad \text{if $\lambda\neq 0$ }, \qquad f_{\epsilon,c}(0)=0.
\end{equation*}
Then
$Q_{\epsilon,c}=f_{\epsilon,c}(Q)$, and hence
\begin{equation*}
    \Sp Q_{\epsilon,c}= f_{\epsilon,c}\left(\Sp Q\right).
\end{equation*}
Set $\mu:=\inf \{|\lambda|; \ \lambda \in \Sp Q, \lambda \neq 0\}=\|Q^{-1}\|^{-1}$. If $\lambda\geq \mu$, then
\[f_{\epsilon,c}(\lambda)=\lambda+\epsilon \lambda+c\lambda^{-n}\geq (1+\epsilon)\lambda\geq (1+\epsilon)\mu.\]
Likewise, if $\lambda \leq -\mu$, then
\[f_{\epsilon,c}(\lambda)=\lambda-\epsilon \lambda-c\lambda^{-n}\leq (1-\epsilon)\lambda\leq  -(1-\epsilon)\mu.\]
Thus,
\begin{equation*}
 \inf \{|\lambda|; \ \lambda \in \Sp Q_{\epsilon,c}, \lambda \neq 0\} \geq (1- |\epsilon|)\mu.
\end{equation*}

Let us now choose $a$ so that $0<a<(1-|\epsilon|)\mu$; this is always possible  if
$\epsilon$ and $a$ are small enough. Then $Q_{\epsilon,c}$ has no non-zero eigenvalues in the interval $[-a,a]$.
Therefore, using~(\ref{eq:Non-generic.Fa}) we see that
\begin{equation*}
    F(Q_{\epsilon,c}+a)=F(Q_{\epsilon,c})+\Pi_{0}(Q_{\epsilon,c}),
\end{equation*}and hence $P_{\epsilon,c,a}$ is equal to 
\begin{align*}
      \left( F(Q_{\epsilon,c}) +\Pi_{0}(Q_{\epsilon,c})\right)^{m+1}(Q_{\epsilon,c}+a)^{m}
    & = \left( F(Q_{\epsilon,c})^{m+1} + \Pi_{0}(Q_{\epsilon,c})\right)(Q_{\epsilon,c}+a)^{m}\\ 
   & =F(Q_{\epsilon,c})^{m+1}(Q_{\epsilon,c}+a)^{m}+a^{m}\Pi_{0}(Q_{\epsilon,c}).
\end{align*}
As~(\ref{Non-generic.spectral-projections-Pec}) and~(\ref{eq:Non-generic.Q-P})
show that $F(Q_{\epsilon,c})=F(Q)=F(P)$ and $\Pi_{0}(Q_{\epsilon,c})=\Pi_{0}(Q)=\Pi_{0}(P)$,
we deduce that
  \begin{equation}
    P_{\epsilon,c,a}= F(P)^{m+1}(Q_{\epsilon,c}+a)^{m}+a^{m}\Pi_{0}(P).
    \label{eq:Non-generic.Peca}
\end{equation}

As shown immediately above Corollary \ref{cor:PsiDOs1-sing-eta}, when $(\epsilon,c,a)$ converges to $(0,0,0)$ the operator
$Q_{\epsilon,c}+a$ converges to $Q$ in $\Psi^{1}(M,E)$. Combining this with~(\ref{eq:Non-generic.Q-P})
and~(\ref{eq:Non-generic.Peca}) shows that $P_{\epsilon,c,a}$ converges to $F(P)^{m+1}Q^{m}=P$. This proves the
first part of the proposition and completes the proof.
\end{proof}

As immediate consequences of Proposition~\ref{prop:higherPsiDOs-sing-eta1} we obtain

\begin{corollary}\label{cor:higherPsiDOs-sing-eta1}
    We always can perturbate $P$ by an arbitrary small selfadjoint element of $\Psi^{m}(M,E)$ in such way that the eta
    function $\eta(P;s)$ and the zeta functions $\zeta_{\updown}(P;s)$ and $\zeta(|P|;s)$ become singular at all
    points of the admissible set~(\ref{eq:general.admissible-set}).
\end{corollary}

\begin{corollary}\label{cor:higherPsiDOs-sing-eta2}
  Generically, the eta and zeta functions of selfadjoint elliptic \psidos\ of positive orders are singular at all
    points of the admissible set~(\ref{eq:general.admissible-set}).
\end{corollary}

\appendix

 \section{Topologies on Spaces of  \psidos}\label{sec:Topology-PsiDOs}
 In this section, for reader's convenience we briefly recall the definition of the standard Fr\'echet space topology of the space
 $\Psi^{m}(M,E)$, $m \in \C$.

 Let $U$ be an open subset of $\Rn$. The space of classical symbols $S^{m}(\URn,M_{r}(\C))$ consists of functions
 $p(x,\xi)\in C^{\infty}(\URn,M_{r}(\C))$ with an asymptotic expansion,
 \begin{equation}
     p(x,\xi) \sim \sum_{j\geq 0} p_{m-j}(x,\xi),
     \label{eq:Appendix.asymptotic-symbols}
 \end{equation}where $p_{m-j}(x,\xi)\in C^{\infty}(\URno,M_{r}(\C))$ is homogeneous of degree $m-j$ with respect to the
 variable $\xi$, i.e.,
 \begin{equation*}
     p_{m-j}(x,\lambda \xi)=\lambda^{m-j}p_{m-j}(x,\xi) \qquad \forall \lambda>0,
 \end{equation*}and the asymptotics is taken in the sense that, for all compacts $K\subset U$, integers $N\geq 1$ and
 multi-orders $\alpha$ and $\beta$, there exists a constant $C_{KN\alpha\beta}>0$ such that, for all $x \in K$ and
 $\xi\in \Rn$ with $|\xi|\geq 1$, we have
 \begin{equation*}
     \biggl| \partial_{x}^{\alpha}\partial_{\xi}^{\beta}\biggl(p(x,\xi)-\sum_{j<N}p_{m-j}(x,\xi)\biggr) \biggr| \leq
     C_{KN\alpha\beta}|\xi|^{m-N-|\beta|}.
 \end{equation*}

 We endow $S^{m}(\URn,M_{r}(\C))$ with the locally convex topological vector space topology defined by the semi-norms,
 \begin{gather*}
     \mathfrak{p}_{KN}(p):= \sup_{|\alpha|+|\beta|\leq N}\sup_{(x,\xi)\in K\times
     \Rn}(1+|\xi|)^{-m+|\beta|}|\partial_{x}^{\alpha}\partial_{\xi}^{\beta}p(x,\xi)|  \\
    \mathfrak{q}_{KN}(p):=\sup_{|\alpha|+|\beta|\leq N}\sup_{\substack{x\in K
      \\ |\xi|\geq 1}}(1+|\xi|)^{-m+|\beta|+N} \biggl|
      \partial_{x}^{\alpha}\partial_{\xi}^{\beta}\biggl(p(x,\xi)-\sum_{j<N}p_{m-j}(x,\xi)\biggr) \biggr|,
 \end{gather*}where $K$ ranges over a (countable) compact exhaustion of $U$ and $N$ ranges over all positive integers. One can check
 that with respect to this topology $S^{m}(\URn,M_{r}(\C))$ is a Fr\'echet space.

 Let $\tau:E_{|V}\stackrel{\sim}{\longrightarrow} V\times \C^{r}$ be a local trivialization of $E$ over an open $V\subset M$
 which is the domain of a local chart $\kappa:V\stackrel{\sim}{\longrightarrow} U\subset \Rn$. We then have
 pushforward and pullback maps,
 \begin{equation*}
     \tau_{*}:C^{\infty}(V,E_{|V}) \rightarrow C^{\infty}(V,\C^{r}) \qquad \text{and} \qquad
      \tau^{*}: C^{\infty}(V,\C^{r})\rightarrow C^{\infty}(V,E_{|V}),
 \end{equation*}
 such that
 \begin{gather*}
   \tau(u(x))=(x,\tau_{*}u(x)) \qquad \forall u \in  C^{\infty}(V,E_{|V}) ,   \\
     \tau^{*}u(x)=\tau^{-1}(x,u(x)) \qquad \forall u \in  C^{\infty}(V,\C^{r}) .
 \end{gather*}

 Let $P\in \Psi^{m}(M,E)$. In the local coordinates and trivialization defined by $\kappa$ and $\tau$, the operator $P$
 corresponds to the operator $P_{\kappa,\tau}\in \Psi^{m}(U,\C^{r})$ such that
 \begin{equation*}
     Pu(x)= \tau^{*}\biggl[ P_{\kappa,\tau}\biggl((\tau_{*}u)\circ \kappa^{-1}\biggr)\biggr] (\kappa(x)) \qquad \forall
     u \in C^{\infty}_{c}(V,E) \ \forall x \in V.
 \end{equation*}

 Let $\varphi$ and $\psi$ be functions in $C^{\infty}_{c}(U)$. Then the Schwartz kernel of the operator $\varphi  P_{\kappa,\tau} \psi$ has compact support, and hence $\varphi
P_{\kappa,\tau} \psi$ is properly supported. Therefore, there exists a unique symbol $p^{\varphi,\psi}_{\kappa,\tau}(P)(x,\xi)\in
 S^{m}(\URn,M_{r}(\C))$ so that
 \begin{equation*}
     \varphi(x) P_{\kappa,\tau} (\psi u) (x)= (2\pi)^{-n} \int e^{ix\cdot \xi} p^{\varphi,\psi}_{\kappa,\tau}(P)(x,\xi).\hat{u}(\xi) d\xi \qquad \forall u \in
     C^{\infty}_{c}(U,\C^{r}).
 \end{equation*}Namely, setting $e_{\xi}(x):=e^{ix\cdot \xi}$, we have
 \begin{equation*}
     p^{\varphi,\psi}_{\kappa,\tau}(P)(x,\xi)(P)(x,\xi)=e^{-ix\cdot \xi}\varphi(x) P_{\kappa,\tau}(\psi e_{\xi})(x) \qquad \forall (x,\xi)\in \URn.
 \end{equation*}
We thus get a linear map,
 \begin{equation}
     \Psi^{m}(M,E) \ni P \longrightarrow p^{\varphi,\psi}_{\kappa,\tau}(P)(x,\xi)\in
 S^{m}(\URn,M_{r}(\C)).
     \label{eq:Appendix.inclusion-Pkt}
 \end{equation}

 The topology of $\Psi^{m}(M,E)$ is the weakest locally convex topological vector space topology that makes continuous
\emph{all} the linear maps~(\ref{eq:Appendix.inclusion-Pkt}) as $(\kappa,\tau)$ ranges  over all pairs where $\kappa$ is a local
 chart for $M$ and $\tau$ is a local trivialization of $E$ over the domain of $\kappa$ and the pair $(\varphi,\psi)$
 range over all pairs of functions in $C^{\infty}_{c}(U)$.
 With respect to this topology
 $\Psi^{m}(M,E)$ is a Fr\'echet space.  In addition, as $M$ is compact, one can check that the standard operations
 with \psidos, such
 that as products, taking adjoints or actions of diffeomorphisms, are all continuous with respect to this topology.

\section{Corrections to Theorem~6.3 of~\cite{Po:SAZFNCR}}\label{sec:eta-odd-class}
As we shall explain Proposition~\ref{prop:1stDO.perturbation} contradicts Theorem~6.3 of~\cite{Po:SAZFNCR}. In this 
appendix, we state and prove the correct version of that statement. 

Let $P:C^{\infty}(M,E)\rightarrow C^{\infty}(M,E)$ be a selfadjoint elliptic \psido\ of order $m$, $m \in \N$, which is odd-class,
in the sense that, in any given local coordinates and trivialization of $E$, the homogeneous components of its symbol
satisfy~(\ref{eq:DO.odd-class-condition}). The class of odd-class \psidos\ is a subalgebra of
$\Psi^{\Z}(M,E)$. In particular, it is invariant  under perturbations by constants. Moreover, any parametrix of an odd-class \psido\
is again an odd-class \psido. Therefore, for all $k \in \Z$, the operator $P^{-k}$ is an odd-class \psido.


Theorem~6.3 of~\cite{Po:SAZFNCR} asserts that if $m$ 
and $n=\dim M$ have opposite parities, then the eta function $\eta(P;s)$
is regular at all integer points.  If $P$ has order $1$ and $n$ is even, then this implies that
the eta function is actually entire. This assertion is clearly contradicted by Proposition~\ref{prop:1stDO.perturbation}.

As we shall now explain, the correct version of Theorem~6.3 of~\cite{Po:SAZFNCR} is

\begin{theorem}\label{thm:odd-class.new}
    \begin{enumerate}
        \item  If $n$ is odd and $m$ is even, then $\eta(P;s)$ is regular at all integer
        points.\smallskip

        \item  If $n$ is even and $m$ is odd, then $\eta(P;s)$ is regular at all even integer
        points.\smallskip

        \item  If $n$ is even and $m=1$, then the singular set of the function $\eta(P;s)$ only contains odd integers~$\leq
        n$.
    \end{enumerate}
\end{theorem}

In other words, Theorem~6.3 of~\cite{Po:SAZFNCR} is true without modifications in odd dimension, but in even dimension we only
have regularity at \emph{even} integer points.

As it turns out, the caveat in the proof of Theorem~6.3 of~\cite{Po:SAZFNCR} comes from using  Proposition~6.2
of~\cite{Po:SAZFNCR}, which relates the eta function $\eta(P;s)$ to the zeta functions $\zeta_{\updown}(P;s)$. Namely,
by Eq.~(6.2) of~\cite{Po:SAZFNCR}, for all $s \in \C$,  we have
 \begin{equation*}
    P_{\scriptscriptstyle{\uparrow}}^{s}-P_{\scriptscriptstyle{\downarrow}}^{s} =
    (1-e^{i\pi s})P_{\scriptscriptstyle{\uparrow}}^{s}-(1-e^{i\pi s})F|P|^{s}.
 \end{equation*}Therefore, we have the equality of meromorphic functions on $\C$,
  \begin{equation}
       \zeta_{\scriptscriptstyle{\uparrow}}(P;s)-\zeta_{\scriptscriptstyle{\downarrow}}(P;s)= (1-e^{-i\pi s})
     \zeta_{\scriptscriptstyle{\uparrow}}(P;s) - (1-e^{-i\pi s})\eta(P;s).
       \label{eq:odd-class.zeta-eta1}
  \end{equation}
  The above equality is the content of the first part of Proposition 6.4 of \cite{Po:SAZFNCR}. However, the other
  parts of this statement about the residues of the eta functions at integer points do not hold in full generality, but they do hold if we restrict ourselves
  to \emph{even} integers. More precisely, we have

  \begin{proposition}\label{prop:odd-class.zeta-eta.new}
  \begin{enumerate}
      \item  If $k$ is an even integer, then
  \begin{equation}
      m.  \lim_{s\rightarrow k}(\zeta_{\scriptscriptstyle{\uparrow}}(P;s)-\zeta_{\scriptscriptstyle{\downarrow}}(P;s)) =
i\pi \Res P^{-k}- i\pi m.\res_{s=k}\eta(P;s).
\label{eq:odd-class.zeta-eta2}
  \end{equation}
      \item  Let $k$ be an even integer such that $\Res P^{-k}=0$, so that both
$\zeta_{\scriptscriptstyle{\uparrow}}(P;s)$ and $\zeta_{\scriptscriptstyle{\downarrow}}(P;s)$
are regular at $s=k$. Then
    \begin{equation*}
        \zeta_{\scriptscriptstyle{\uparrow}}(P;k)=\zeta_{\scriptscriptstyle{\downarrow}}(P;k) \Longleftrightarrow
        \text{$\eta(P;s)$ is regular at $s=k$}.
     \end{equation*}
  \end{enumerate}
\end{proposition}
\begin{proof}
   As the 2nd part is an immediate consequence of the first part, we only have to prove the latter. Suppose that $k$ is an even integer.
   Then $1-e^{-i\pi s}\sim i\pi(s-k)$ near $s=k$, and hence
\begin{gather*}
 \lim_{s\rightarrow k}  (1-e^{-i\pi s})
     \zeta_{\scriptscriptstyle{\uparrow}}(P;s)= i\pi \Res_{s=k} \zeta_{\scriptscriptstyle{\uparrow}}(P;s)= i\pi  m^{-1}\Res P^{-k},   \\
     \lim_{s\rightarrow k}(1-e^{-i\pi s})\eta(P;s)= i\pi \Res_{s=k} \eta(P;s).
\end{gather*}
Combining this with~(\ref{eq:odd-class.zeta-eta1}) then gives~(\ref{eq:odd-class.zeta-eta2}). This proves the first
part and completes the proof.
\end{proof}


We are now in a position to prove Theorem~\ref{thm:odd-class.new}.

\begin{proof}[Proof of Theorem~\ref{thm:odd-class.new}]
    Suppose that $n$ is odd and $m$ is even, and let $k \in \Z$. As mentioned above $P^{-k}$ is an odd-class
    \psido. Moreover, as alluded to in the proof of Lemma~\ref{lem:DO.singularities-odd-dimension}, in odd dimensions the noncommutative residue vanishes
    on all odd-class \psidos. In particular, we see that
    \begin{equation}
        \Res P^{-k}=0.
        \label{eq:odd-class.resPk}
    \end{equation}

    Suppose that $k$ is even. Then, by Theorem~5.1
    of~\cite{Po:SAZFNCR},
    \begin{equation*}
         \lim_{s\rightarrow k}(\zeta_{\scriptscriptstyle{\uparrow}}(P;s)-\zeta_{\scriptscriptstyle{\downarrow}}(P;s)) = 0.
    \end{equation*}Combining this with (\ref{eq:odd-class.zeta-eta2}) and~(\ref{eq:odd-class.resPk})
    shows that $\eta(P;s)$ is regular at $s=k$ when $k$ is even.

    Assume that $k$ is odd. Thanks to~(\ref{eq:Background.Residues-eta}) we have
    \begin{equation*}
        \Res_{s=k}\eta(P;s)=m^{-1}\Res \left[F|P|^{-k}\right]=m^{-1}\Res \left[F^{k+1}P^{-k}\right].
    \end{equation*}
    Moreover, as $k+1$ is even, Eq.~(\ref{eq:DiffO-signPk}) implies that $F^{k+1}P^{-k}=P^{-k} \bmod \Psi^{-\infty}(M,E)$. Thus,
    \begin{equation*}
         \Res_{s=k}\eta(P;s)=m^{-1}\Res P^{-k}=0.
    \end{equation*}This shows that $\eta(P;s)$ is regular at $s=k$, even when $k$ is odd.

    Suppose now that $n$ is even and $m$ is odd, and let $k$ be an \emph{even} integer. Thanks to Theorem~5.2
    of~\cite{Po:SAZFNCR} we have
        \begin{equation*}
        m. \lim_{s\rightarrow k}(\zeta_{\scriptscriptstyle{\uparrow}}(P;s)-\zeta_{\scriptscriptstyle{\downarrow}}(P;s)) =
         i\pi \Res P^{-k}.
    \end{equation*}Combining this with~(\ref{eq:odd-class.zeta-eta2})
    proves that $\Res_{s=k}\eta(P;s)=0$, that is, the function $\eta(P;s)$ is regular at $s=k$.

    If we further assume that $P$ has order $1$, then the admissible set~(\ref{eq:Non-generic.singular-set}) at which
    the function $\eta(P;s)$ may be singular only contains integers~$\leq n$. Since
    we know that $\eta(P;s)$ is regular at all even integer points, we see that the singularities of $\eta(P;s)$ can
    only occur at odd integers~$\leq n$. The proof of  Theorem~\ref{thm:odd-class.new} is thus complete.
\end{proof}

\end{document}